\documentclass[12pt]{amsart}  

\usepackage{amsmath,amsthm,amssymb,amscd,mathdots,enumerate,float,shuffle,hyperref}
\usepackage{tikz}

\usetikzlibrary{decorations.pathreplacing,angles,quotes}
\allowdisplaybreaks

\newtheorem{theorem}{Theorem}[section]
\newtheorem{proposition}[theorem]{Proposition}

\newtheorem{lemma}[theorem]{Lemma}
\newtheorem{corollary}[theorem]{Corollary}

\numberwithin{equation}{section}

\theoremstyle{definition}
\newtheorem{definition}[theorem]{Definition}
\newtheorem{example}[theorem]{Example}
\newtheorem{remark}[theorem]{Remark}
\newcommand{\Q}{\mathbb{Q}}
\newcommand{\Ha}{\mathbb{H}}

\newcommand{\R}{\mathbb{R}}
\newcommand{\C}{\mathbb{C}}

\newcommand{\Hom}{\operatorname{Hom}}

\newcommand{\dz}{\mathcal{D}} 
\newcommand{\de}{\mathcal{E}}

\newcommand{\Genz}{\mathfrak{G}}
\newcommand{\eenz}{\mathfrak{E}}
\newcommand{\genz}{\mathfrak{g}}
\newcommand{\penz}{\mathfrak{P}}
\newcommand{\benz}{\mathfrak{b}}

\DeclareRobustCommand{\gp}{\penz\!\genfrac{(}{)}{0pt}{}}
\newcommand\quotient[2]{
	\mathchoice
	{% \displaystyle
		\text{\raise1ex\hbox{$#1$}\Big/\lower1ex\hbox{$#2$}}%
	}
	{% \textstyle
		#1\,/\,#2
	}
	{% \scriptstyle
		#1\,/\,#2
	}
	{% \scriptscriptstyle  
		#1\,/\,#2
	}
}

\DeclareRobustCommand{\bi}{\genfrac{(}{)}{0pt}{}}

\title[q-analogues of MZV and the formal double Eisenstein space]{q-analogues of multiple zeta values \\and the formal double Eisenstein space}

\author{Henrik Bachmann}
\address{Graduate School of Mathematics,  Nagoya University, Nagoya, Japan.}
\email{henrik.bachmann@math.nagoya-u.ac.jp}

\subjclass[2010]{Primary 11M32; Secondary 	11F11}

\begin{document}
\date{\today}
\maketitle

\begin{abstract} 
In this survey article, we discuss the algebraic structure of $q$-analogues of multiple zeta values, which are closely related to derivatives of Eisenstein series. Moreover, we introduce the formal double Eisenstein space, which generalizes the formal double zeta space of Gangl, Kaneko, and Zagier. Using the algebraic structure of $q$-analogues of multiple zeta values, we will present a realization of this space. As an application, we will obtain purely combinatorial proofs of identities among (quasi-)modular forms. 
\end{abstract}

\section{Introduction}
In this note, we discuss the algebraic structure of certain $q$-analogues of multiple zeta values and their connection to modular forms. The purpose of this survey article is to give a self-contained summary of some of the results obtained in \cite{B1} and \cite{B2} and to motivate the works in progress \cite{BB}, \cite{BKM} and \cite{BIM}. Some parts of those results were also announced in \cite{BBK}. In particular, we give an explanation of the formal double Eisenstein space, which will be introduced and discussed in detail in \cite{BKM} and which can be seen as a generalization of the formal double zeta spaces introduced in \cite{GKZ}. We will see that understanding the algebraic structure of $q$-analogues of multiple zeta values allows us to write down an explicit realization of this space. As an application of this, we will give some purely combinatorial proofs of identities among (quasi-)modular forms such as the Ramanujan differential equations. \\

Multiple zeta values are defined for $k_1\geq 2, k_2,\dots,k_r \geq 1$ by
\begin{align}\label{eq:defmzv}
	\zeta(k_1,\ldots,k_r)=\sum_{m_1>\cdots>m_r>0 }\frac{1}{m_1^{k_1}\cdots m_r^{k_r}}\,.
\end{align}
These values are known to satisfy a collection of relations called the double shuffle relations. For example, in smallest depth and $k_1,k_2 \geq 2$ these are given by 
\begin{align}\label{eq:dzvds}\begin{split}
\zeta(k_1) \zeta(k_2) &= \zeta(k_1,k_2) + \zeta(k_2, k_1) + \zeta(k_1+k_2) \\
&= \sum_{j=2}^{k_1+k_2-1} \left( \binom{j-1}{k_1-1} + \binom{j-1}{k_2-1} \right) \zeta(j, k_1+k_2-j)\,.
\end{split}
\end{align}
The first one, called the stuffle product, is a direct consequence of the definition \eqref{eq:defmzv}. The second one, called the shuffle product, can be proven either by using an iterated integral expression for multiple zeta values or by using partial fraction decomposition. Motivated by the relations \eqref{eq:dzvds} Gangl, Kaneko and Zagier defined in \cite{GKZ} the \emph{formal double zeta space} for $k\geq 1$ by the $\Q$-vector space
\begin{align*}
	\dz_k = \quotient{\big \langle Z_k , Z_{k_1,k_2}, P_{k_1,k_2} \mid k_1 + k_2 = k, k_1,k_2 \geq 1 \big \rangle_\Q}{\eqref{eq:dzrel} }\,,        
\end{align*}
which is spanned by formal symbols $Z_k , Z_{k_1,k_2}, P_{k_1,k_2}$ satisfying the relations 
\begin{align}
	\begin{split}
		\label{eq:dzrel} 
		P_{k_1,k_2} &= Z_{k_1, k_2} +  Z_{k_2, k_1} + Z_{k_1+k_2} \\
		&= \sum_{j=1}^{k_1+k_2-1} \left( \binom{j-1}{k_1-1} + \binom{j-1}{k_2-1} \right) Z_{j, k_1+k_2-j}\,.
	\end{split}
\end{align}
In \cite{GKZ} different realizations of this space are given, where for a $\Q$-vector space $A$ a realization of $\dz_k$ in $A$ is defined as an element in $\Hom_\Q(\dz_k,A)$. By \eqref{eq:dzvds} together with some regularization one can show that for $A=\R$ we obtain a realization of $\dz_k$ by 
\begin{align}\begin{split}\label{eq:mzvrealization}
	Z_{k} &\longmapsto \begin{cases}
		\zeta(k) & k\geq 3\\
		0 & k=1,2
	\end{cases}\,,\\
	Z_{k_1,k_2} &\longmapsto  \begin{cases}
		\zeta(k_1,k_2) & k_1 \geq 2 \\
		-\zeta(k_2,1)-\zeta(k_2+1) & k_1 = 1, k_2\geq 2\\
		0 & k_1 = k_2 = 1 
	\end{cases}\,,\\
	P_{k_1,k_2} &\longmapsto \begin{cases}
		\zeta(k_1) \zeta(k_2) & k_1,k_2 \neq 1 \\
		0 & k_1 = 1 \vee k_2=1 
	\end{cases}\,.
	\end{split}
\end{align}
The Riemann zeta values $\zeta(k)$ also appear as the constant term of the Eisenstein series defined for $k\geq 2$ by
\begin{align}\label{def:gk}
G_k(\tau) = \zeta(k) + \frac{(-2\pi i)^k}{(k-1)!}\sigma_{k-1}(n)q^n\,,
\end{align}
where $\tau \in \Ha =\{ z \in \C \mid \operatorname{Im}(z) > 0\}$, $q=e^{2\pi i \tau}$ and $\sigma_{k-1}(n)=\sum_{d\mid n} d^{k-1}$.
For even $k\geq 4$ these are the first non-trivial examples of modular forms of weight $k$ for the full modular group. In \cite{GKZ} it was observed, that the relations \eqref{eq:dzvds} can be lifted to the space of holomorphic functions, such that the Riemann zeta values gets replaced by the Eisenstein series and the double zeta values by so-called double Eisenstein series $G_{k_1,k_r}$. In other words they showed that for $k\geq 4$ there exists a realization of $\dz_k$ in the space of holomorphic functions in the upper half plane with $Z_k \mapsto G_k$ (where $G_1$ is given by \eqref{def:gk} without the constant term), $	Z_{k_1,k_2} \mapsto  G_{k_1,k_r}$ and
\begin{align}\label{eq:pimageeisen}
P_{k_1,k_2} &\longmapsto G_{k_1} G_{k_2} + \frac{\delta_{k_1,2}}{2 k_2} G'_{k_2} + \frac{\delta_{k_2,2}}{2 k_1} G'_{k_1}\,.
\end{align}
Taking the constant term of this realization gives the multiple zeta value realization \eqref{eq:mzvrealization}.
But in contrast to  \eqref{eq:mzvrealization} the image of $P_{k_1,k_2}$ in \eqref{eq:pimageeisen} is not given exactly by the product of the depth one objects, since also derivatives of Eisenstein series appear. This has to do with the fact that $G_2$ is not modular, and it is one of the motivations for the formal double Eisenstein space, which will be introduced in \cite{BKM} and which we will explain at the end of this note. In this space, we will also consider "formal derivatives" of our objects. As it will turn out, this gives an analogue of the formal double zeta space, which is better suited for Eisenstein series in the sense that we will find a realization of this space where the formal symbol of the product gets mapped exactly to the product of Eisenstein series. It can also be seen as a generalization of the work of Gangl, Kaneko, Zagier from modular forms to quasi-modular forms.  

\subsection*{Acknowledgment}
The author would like to thank the organizers of the Waseda number theory Conference 2021 for giving him the opportunity to present parts of the results of \cite{BB}, \cite{BIM} and \cite{BKM}  during the conference. These projects are partially supported by JSPS KAKENHI Grant 19K14499.

\section{$q$-analogues of multiple zeta values}
This section recalls the $q$-analogues of multiple zeta values, which were introduced in \cite{B1}. Parts of the results are contained in \cite{BK} and detailed information can be found in \cite{B2}. Most of the time we will only consider the depth two case, but everything done in this section also has a generalization to arbitrary depths. 

\begin{definition}
A \emph{modified q-analogue of weight $k$} of a complex number $c \in \C$ is a q-series $f \in \C[[q]]$, such that
	\begin{align*}
		\lim_{q\rightarrow 1} (1-q)^k f = c\,.
	\end{align*}
\end{definition}
Here and in the following by $q\rightarrow 1$, we always mean the limit of $q$ on the real axis with $0<q<1$. Usually, our objects are just formal $q$-series in $\Q[[q]]$, but all appearing $q$-series in this note can also be seen as holomorphic functions in the unit disc ($|q|<1$) or as holomorphic functions in the upper half-plane by setting $q=e^{2\pi i \tau}$. We will not distinguish these three points of view in the following. 
\begin{proposition} Let $f(\tau)= a_0 + \sum_{n\geq 1} a_n q^n$ be a modular form of weight $k$.Then $f$ is a modified q-analogue of $(2\pi i)^k a_0$.
\end{proposition}
\begin{proof} If $f$ is modular of weight $k$ then $f(-\frac{1}{\tau}) = \tau^k f(\tau)$, i.e.
		\begin{align*}
	\lim_{q\rightarrow 1} (1-q)^k f(q) &= \lim_{\tau\rightarrow 0} ((2\pi i \tau)^k +  O(\tau^{k+1})) f(\tau) = \lim_{\tau\rightarrow 0} (2\pi i)^k f\left(-\frac{1}{\tau}\right) \\&= \lim_{\tau \rightarrow i\infty} (2\pi i)^k f(\tau) =  (2\pi i)^k a_0\,.
\end{align*}
\end{proof}

In particular the normalized Eisenstein series 
\begin{align}\label{eq:defgtilde}
	\widetilde{G}_k(\tau) := \frac{1}{(2\pi i)^k} G_k(\tau) = -\frac{B_k}{2k!}+ \frac{1}{(k-1)!}\sum_{n>0} \sigma_{k-1}(n)q^n 
\end{align}
are modified $q$-analogue (of weight $k$) of $\zeta(k)$ for even $k \geq 4$. But we will see that we have for any $k\geq 2$
\begin{align*}
	\lim_{q\rightarrow 1}  (1-q)^k	\underbrace{\frac{1}{(k-1)!}\sum_{n>0} \sigma_{k-1}(n)q^n}_{g(k) :=}  =  \zeta(k)\,.
\end{align*}
In the following we will consider a multiple version of the $q$-series $g(k)= \sum_{\substack{m>0\\n>0}} \frac{n^{k-1}}{(k-1)!} q^{mn}$.
\begin{definition}\label{def:singleg}For $k_1,\dots k_r \geq 1$ we define the $q$-series $g(k_1,\dots,k_r) \in \Q[[q]]$ by
	\begin{align*}
		g(k_1,\dots,k_r)= \sum_{\substack{m_1 > \dots > m_r > 0\\ n_1, \dots , n_r > 0}} \frac{n_1^{k_1-1}}{(k_1-1)!} \dots \frac{n_r^{k_r-1}}{(k_r-1)!}  q^{m_1 n_1 + \dots + m_r n_r } \,.
	\end{align*}
\end{definition}
Notice that for $r=1$ these are, up to the constant term, exactly the Eisenstein series $\widetilde{G}_k$. Moreover, they are modified q-analogues of multiple zeta values. 
\begin{proposition}\label{prop:garemodqanalog}
	For $k_1 \geq 2$, $k_2,\dots,k_r \geq 1$ the $g(k_1,\dots,k_r)$ are modified q-analogues of $\zeta(k_1,\dots,k_r)$, i.e. 
	\begin{align*}
	\lim\limits_{q\rightarrow 1} (1-q)^{k_1+\dots+k_r} g(k_1,\dots,k_r) = \zeta(k_1,\dots,k_r)\,.
\end{align*}	
\end{proposition}
\begin{proof} First rewrite 
	\begin{align*}
	g(k_1,\dots,k_r)= \sum_{m_1 > \dots > m_r > 0} \frac{P_{k_1}(q^{m_1})}{(1-q^{m_1})^{k_1}} \dots  \frac{P_{k_r}(q^{m_r})}{(1-q^{m_r})^{k_r}}\,,
\end{align*}
where the polynomials $P_k$ are defined by $\frac{P_k(X)}{(1-X)^k}:= \sum_{n>0}\frac{n^{k-1}}{(k-1)!} X^{k-1} $. These are, up to a normalization, the \emph{Eulerian polynomials} and they satisfy $P_k(1) = 1$ from which we get after using some elementary convergence criteria that
	\begin{align*}
	\lim\limits_{q\rightarrow 1} (1-q)^{k_1+\dots+k_r} g(k_1,\dots,k_r) = 	\lim\limits_{q\rightarrow 1} \sum_{m_1 > \dots > m_r > 0} \prod_{j=1}^r \left( \frac{1-q}{1-q^{m_j}}\right)^{k_j} P_{k_j}(q^{m_j}) = \sum_{m_1 > \dots > m_r > 0}   \prod_{j=1}^r \frac{1}{m_j^{k_j}} \,.
\end{align*}
\end{proof}

\begin{proposition} \label{prop:stuffle1}For $k_1,k_2 \geq 1$ we have
\begin{align}\label{eq:stuffle1}
	g(k_1)g(k_2) = g(k_1,k_2)+g(k_2,k_1)+g(k_1+k_2)+ \sum_{j=1}^{k_1+k_2-1} \lambda^j_{k_1,k_2}  \,g(j)\,,
\end{align}
where the rational numbers $\lambda^j_{k_1,k_2}$ are given by
{\small
\begin{align}\label{eq:deflambda}
	\lambda^j_{k_1,k_2}  = \left((-1)^{k_1-1} \binom{k_1+k_2-1-j}{k_2 -j} + (-1)^{k_2-1} \binom{k_1+k_2-1-j}{k_1-j}  \right) \frac{B_{k_1+k_2-j}}{(k_1+k_2-j)!} \,.
\end{align}}
\end{proposition}
\begin{proof}The proof uses the generating function $
	L_m(X) = \sum_{k\geq 1} \frac{P_k(q^m)}{(1-q^m)^k} X^k = \frac{e^X q^m}{1- e^X q^m}
$. By direct calculation one then checks that
\begin{align*}
	L_m(X) L_m(Y) =  \frac{1}{e^{X-Y}-1} L_m(X) + \frac{1}{e^{Y-X}-1} L_m(Y),
\end{align*}
and uses $\sum_{n=0}^\infty B_n \frac{x^n}{n!} = \frac{x}{e^x-1}$ to show
\begin{align*}
    \frac{P_{k_1}(q^m)}{(1-q^m)^{k_1}}  \frac{P_{k_2}(q^m)}{(1-q^m)^{k_2}} &=  \frac{P_{k_1+k_2}(q^m)}{(1-q^m)^{k_1+k_2}} + \sum_{j=1}^{k_1+k_2-1} \lambda^j_{k_1,k_2} \frac{P_{j}(q^m)}{(1-q^m)^{j}}\,.
\end{align*}
The statement then follows since 
\begin{align*}
	g(k_1)g(k_2) &= \sum_{m_1>0}   \frac{P_{k_1}(q^{m_1})}{(1-q^{m_1})^{k_1}} \sum_{m_2>0}   \frac{P_{k_2}(q^{m_2})}{(1-q^{m_2})^{k_2}} \\
	&= \left(\sum_{m_1>m_2>0} +\sum_{m_2>m_1>0}+ \sum_{m_1=m_2>0}   \right)\frac{P_{k_1}(q^{m_1})}{(1-q^{m_1})^{k_1}}\frac{P_{k_2}(q^{m_2})}{(1-q^{m_2})^{k_2}}\,. 
\end{align*}
\end{proof}
Except for the lower weight terms the product \eqref{eq:stuffle1} looks exactly like the stuffle product $\zeta(k_1) \zeta(k_2) = \zeta(k_1,k_2) + \zeta(k_2, k_1) + \zeta(k_1+k_2)$ of zeta values. Notice that this formula for zeta values follows from Proposition \ref{prop:garemodqanalog} for $k_1,k_2\geq 2$ by multiplying  \eqref{eq:stuffle1} with $(1-q)^{k_1+k_2}$ and then taking the limit $q\rightarrow 1$ since the lower weight terms vanish. So a natural question is if there is also an analogue of the shuffle product for the $q$-series $g$. Before answering this question we consider the operator $q \frac{d}{dq}$ (which corresponds to $\frac{1}{2\pi i}\frac{d}{d \tau}$ for $q=e^{2\pi i \tau}$).
\begin{align*}
	q  \frac{d}{dq} g_k(q) = q  \frac{d}{dq} \sum_{\substack{m>0\\ n>0}} \frac{n^{k-1}}{(k-1)!} q^{mn} =  \sum_{\substack{m>0\\ n>0}}  \frac{m n^{k}}{(k-1)!} q^{mn}\,.
\end{align*}
We see that after taking the derivative we also have a $m$ appearing in the numerator. Moreover if we would take the $d$-th derivative we would get 
\begin{align*}
	\left( q  \frac{d}{dq} \right)^d g_k(q) =  \sum_{\substack{m>0\\ n>0}}  \frac{m^d n^{k+d-1}}{(k-1)!} q^{mn}\,.
\end{align*}
This leads us to define $g$ in a more general way. 
\begin{definition}For $k_1,\dots k_r \geq 1, d_1,\dots, d_r \geq 0$ define the $q$-series 
	\begin{align*}
		g\bi{k_1,\dots,k_r}{d_1,\dots,d_r}= \sum_{\substack{m_1 > \dots > m_r > 0\\ n_1, \dots , n_r > 0}} \frac{m_1^{d_1} n_1^{k_1-1}}{(k_1-1)!} \dots \frac{m_r^{d_r} n_r^{k_r-1}}{(k_r-1)!}  q^{m_1 n_1 + \dots + m_r n_r } \,.
	\end{align*}
	We say that this has \emph{weight} $k_1 + \dots + k_r + d_1 + \dots + d_r$ and \emph{depth} $r$.
\end{definition}
Clearly these generalize the $q$-series in Definition \ref{def:singleg} since we have
\begin{align*}
    g(k_1,\dots,k_r) = g\bi{k_1,\dots,k_r}{0,\dots,0}\,.
\end{align*}
Notice that again for $r=1$ these $q$-series are essentially the derivatives of Eisenstein series since for $k>d$ we have
\begin{align*}
     g\bi{k}{d} =  \frac{(k-d-1)!}{(k-1)!} \left(q \frac{d}{dq}\right)^d \widetilde{G}_{k-d} +\delta_{d,0} \frac{B_k}{2k!}\,.
\end{align*}
With the same idea as before it is easy to see that we have 
\begin{align*}
	q  \frac{d}{dq} g\bi{k_1,\dots,k_r}{d_1,\dots,d_r} = \sum_{j=1}^r 	k_j g\bi{k_1,\dots.k_j+1,\dots,k_r}{d_1,\dots,d_j+1,\dots,d_r}\,.
\end{align*}
But the main reason for introducing this more general definition is not that we can talk about derivatives, but rather we will see that this enables us to define an analogue of the shuffle product. Before doing this we will see that the formula for the stuffle product (Proposition \ref{prop:garemodqanalog}) generalizes easily. 
\begin{proposition}[Stuffle product analogue] \label{prop:stufflebig}
     For $k_1,k_2 \geq 1, d_1,d_2 \geq 0$ we have
\begin{align}
	g\bi{k_1}{d_1} g\bi{k_2}{d_2}  = &g\bi{k_1,k_2}{d_1,d_2} +g\bi{k_2,k_1}{d_2,d_1} +g\bi{k_1+k_2}{d_1+d_2}+ \sum_{j=1}^{k_1+k_2-1}\lambda^j_{k_1,k_2} \, g\bi{j}{d_1+d_2} \,,
\end{align}
where the $\lambda^j_{k_1,k_2}$ are given by \eqref{eq:deflambda}.
\end{proposition}
\begin{proof} Similar proof as Proposition \ref{prop:stuffle1} since the exponents of the $m_j$ just add up.\end{proof}

Let $\operatorname{Part}_r(N)$ denote the set of partitions of $N$ made out of $r$ different parts. For example, $11=4+4+1+1+1$ is a partition of $N=11$ made out of $r=2$ different parts $(m_1,m_2)=(4,1)$ with multiplicities $(n_1,n_2)=(2,3)$, i.e. $N = m_1 n_1 + m_2 n_2$.  Any element $\lambda \in \operatorname{Part}_r(N)$ can be represented by a Young diagram 
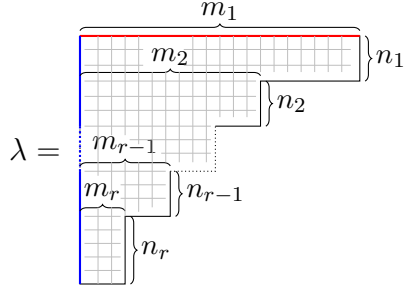
\vspace{-1cm}
\begin{figure}[H] 
	\begin{center}
	\begin{tikzpicture}[scale=0.6]

\draw (-1,3) node{{$\lambda = $}};
\draw (0,2.5) -- (0,0) -- (1,0) -- (1,1.5) -- (2,1.5) -- (2,2.5);
\draw [densely dotted] (0,2.5) -- (0,3.5);
\draw [densely dotted] (2,2.5) -- (3,2.5) -- (3,3.5);
\draw (3,3.5) -- (4,3.5) -- (4,4.5) -- (6.2,4.5) -- (6.2,5.5) -- (0,5.5) -- (0,3.5);

\draw [thick, red] (0,5.5) -- (6.2,5.5); 
\draw [thick, blue] (0,5.5) -- (0,3.5); 
\draw [thick, blue, densely dotted] (0,2.5) -- (0,3.5);
\draw [thick, blue] (0,0) -- (0,2.5);

\draw[decoration={brace,raise=2pt},decorate] (0,5.5) -- node[above=2pt] {\small $m_1$} (6.2,5.5);
\draw[decoration={brace,raise=2pt},decorate] (0,4.5) -- node[above=2pt] {\small $m_2$} (4,4.5);
\draw[decoration={brace,raise=2pt},decorate] (0,2.5) -- node[above=2pt] {\small $m_{r-1}$} (2,2.5);
\draw[decoration={brace,raise=2pt},decorate] (0,1.5) -- node[above=2pt] {\small $m_{r}$} (1,1.5);

\draw[decoration={brace,raise=2pt},decorate] (6.2,5.5) -- node[right=2pt] {\small $n_{1}$} (6.2,4.5);
\draw[decoration={brace,raise=2pt},decorate] (4,4.5) -- node[right=2pt] {\small $n_{2}$} (4,3.5);
\draw[decoration={brace,raise=2pt},decorate] (2,2.5) -- node[right=2pt] {\small $n_{r-1}$} (2,1.5);
\draw[decoration={brace,raise=2pt},decorate] (1,1.5) -- node[right=2pt] {\small $n_{r}$} (1,0);

\draw [lightgray,ultra thin] (4,5.5) -- (4,4.7);
\draw [lightgray,ultra thin] (4.3,5.5) -- (4.3,4.7);
\draw [lightgray,ultra thin] (4.6,5.5) -- (4.6,4.7);
\draw [lightgray,ultra thin] (4.9,5.5) -- (4.9,4.7);
\draw [lightgray,ultra thin] (5.2,5.5) -- (5.2,4.7);
\draw [lightgray,ultra thin] (5.5,5.5) -- (5.5,4.7);
\draw [lightgray,ultra thin] (5.8,5.5) -- (5.8,4.7);
\draw [lightgray,ultra thin] (3.7,5.5) -- (3.7,3.7);
\draw [lightgray,ultra thin] (3.4,5.5) -- (3.4,3.7);
\draw [lightgray,ultra thin] (3.1,5.5) -- (3.1,3.7);
\draw [lightgray,ultra thin] (2.8,5.5) -- (2.8,2.7);
\draw [lightgray,ultra thin] (2.5,5.5) -- (2.5,2.7);
\draw [lightgray,ultra thin] (2.2,4.7) -- (2.2,2.7);
\draw [lightgray,ultra thin] (1.9,4.7) -- (1.9,2.7);
\draw [lightgray,ultra thin] (1.6,4.7) -- (1.6,1.7);
\draw [lightgray,ultra thin] (1.3,5.5) -- (1.3,3.1);
\draw [lightgray,ultra thin] (1,5.5) -- (1,3.3);
\draw [lightgray,ultra thin] (0.7,5.5) -- (0.7,3.3);
\draw [lightgray,ultra thin] (0.4,5.5) -- (0.4,2.4);
\draw [lightgray,ultra thin] (1.3,2.7) -- (1.3,1.6);
\draw [lightgray,ultra thin] (1,2.7) -- (1,1.5);
\draw [lightgray,ultra thin] (0.7,2.7) -- (0.7,2);
\draw [lightgray,ultra thin] (0.4,1.7) -- (0.4,0);
\draw [lightgray,ultra thin] (0.7,1.7) -- (0.7,0);

\draw [lightgray,ultra thin] (0.1,0.3) -- (0.9,0.3);
\draw [lightgray,ultra thin] (0.1,0.6) -- (0.9,0.6);
\draw [lightgray,ultra thin] (0.1,0.9) -- (0.9,0.9);
\draw [lightgray,ultra thin] (0.1,1.2) -- (0.9,1.2);
\draw [lightgray,ultra thin] (0.1,1.5) -- (0.9,1.5);

\draw [lightgray,ultra thin] (0.9,1.9) -- (1.9,1.9);
\draw [lightgray,ultra thin] (0.9,2.2) -- (1.9,2.2);
\draw [lightgray,ultra thin] (0.1,2.5) -- (1.9,2.5);

\draw [lightgray,ultra thin] (0.1,3.5) -- (2.9,3.5);
\draw [lightgray,ultra thin] (1.1,3.2) -- (2.9,3.2);
\draw [lightgray,ultra thin] (1.8,2.9) -- (2.9,2.9);

\draw [lightgray,ultra thin] (0.1,4.5) -- (3.9,4.5);
\draw [lightgray,ultra thin] (0.1,4.16) -- (3.9,4.16);
\draw [lightgray,ultra thin] (0.1,3.85) -- (3.9,3.85);

\draw [lightgray,ultra thin] (0.1,4.85) -- (1.4,4.85);
\draw [lightgray,ultra thin] (0.1,5.2) -- (1.4,5.2);
\draw [lightgray,ultra thin] (2.3,5.2) -- (6,5.2);
\draw [lightgray,ultra thin] (2.3,4.85) -- (6,4.85);

\end{tikzpicture}
	\end{center}
	\vspace{-0.7cm}
\end{figure}
\noindent where $N = m_1 n_1 + \dots + m_r n_r$ and $m_1 > \dots > m_r > 0$, $n_1, \dots, n_r > 0$. In particular we can write the coefficient of $q^N$ in the $q$-series $g$ as a sum over all elements in $\operatorname{Part}_r(N)$ 
\begin{align*}
	g\bi{k_1,\dots,k_r}{d_1,\dots,d_r}= \sum_{\substack{m_1 > \dots > m_r > 0\\ n_1, \dots , n_r > 0}} \underbrace{\frac{m_1^{d_1} n_1^{k_1-1}}{(k_1-1)!} \dots \frac{m_r^{d_r} n_r^{k_r-1}}{(k_r-1)!} }_{f(\lambda):=} q^{m_1 n_1 + \dots + m_r n_r }  = \sum_{N>0} \left( \sum_{\lambda \in \operatorname{Part}_r(N)} f(\lambda) \right) q^N\,,
\end{align*}
where $f: \operatorname{Part}_r(N) \rightarrow \Q$ sends a partition $\lambda$, given as above, to a polynomial in the $m_j$ and $n_j$ depending on the $k_j$ and $d_j$. On $\operatorname{Part}_r(N)$ we have an involution given by the conjugation $\rho$ of Young diagrams, which correspond to reflecting the Young diagram along the diagonal 
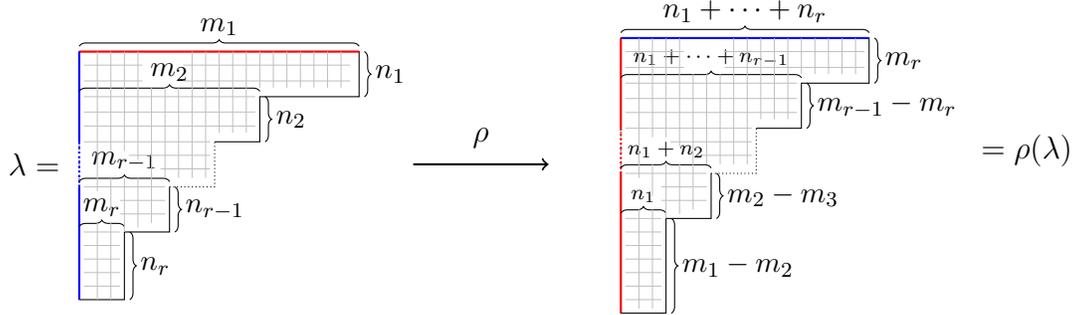
\vspace{-1cm}
\begin{figure}[H] 
	\begin{center}
		\begin{tikzpicture}[scale=0.6]

% \filldraw [draw=none,fill=black!5] ({-cos(60)},2) -- ({cos(60)},2) -- ({cos(60)},{sin(60)}) 
%      arc[start angle=60,end angle=120,radius=1]--({-cos(60)},2);
% \draw [thin]  ({cos(60)},2) -- ({cos(60)},{sin(60)}) 
%      arc[start angle=60,end angle=120,radius=1]--({-cos(60)},2);

% \filldraw ({-cos(60)},{sin(60)}) circle (0.5pt);
%  \draw  ({-cos(60)},{sin(60)}) node[below]{$\omega$};

% \filldraw (0,1) circle (0.5pt);
%  \draw  (0,1) node[below]{$i$};

% \filldraw ({cos(60)},{sin(60)}) circle (0.5pt);
%  \draw  ({cos(60)-0.05},{sin(60)}) node[below]{$-\overline{\omega}$};

% \draw [ultra thin,densely dashed] (-1,0) arc[start angle=180,delta angle=-180,radius=1];

%   \draw (0,1.5) node{{\Large $\mathcal{F}$}};
%   \draw (1,2pt) -- (1,-2pt) node[below]{$1$};
%   \draw (0.5,2pt) -- (0.5,-2pt) node[below]{$\frac{1}{2}$};
%   \draw (0,2pt) -- (0,-2pt) node[below]{$0$};
%   \draw (-0.5,2pt) -- (-0.5,-2pt) node[below]{$-\frac{1}{2}$};
%   \draw (-1,2pt) -- (-1,-2pt) node[below]{$-1$};

 \draw (-1,3) node{{$\lambda = $}};
\draw (0,2.5) -- (0,0) -- (1,0) -- (1,1.5) -- (2,1.5) -- (2,2.5);
\draw [densely dotted] (0,2.5) -- (0,3.5);
\draw [densely dotted] (2,2.5) -- (3,2.5) -- (3,3.5);
\draw (3,3.5) -- (4,3.5) -- (4,4.5) -- (6.2,4.5) -- (6.2,5.5) -- (0,5.5) -- (0,3.5);

\draw [thick, red] (0,5.5) -- (6.2,5.5); 
\draw [thick, blue] (0,5.5) -- (0,3.5); 
\draw [thick, blue, densely dotted] (0,2.5) -- (0,3.5);
\draw [thick, blue] (0,0) -- (0,2.5);

\draw[decoration={brace,raise=2pt},decorate] (0,5.5) -- node[above=2pt] {\small $m_1$} (6.2,5.5);
\draw[decoration={brace,raise=2pt},decorate] (0,4.5) -- node[above=2pt] {\small $m_2$} (4,4.5);
\draw[decoration={brace,raise=2pt},decorate] (0,2.5) -- node[above=2pt] {\small $m_{r-1}$} (2,2.5);
\draw[decoration={brace,raise=2pt},decorate] (0,1.5) -- node[above=2pt] {\small $m_{r}$} (1,1.5);

\draw[decoration={brace,raise=2pt},decorate] (6.2,5.5) -- node[right=2pt] {\small $n_{1}$} (6.2,4.5);
\draw[decoration={brace,raise=2pt},decorate] (4,4.5) -- node[right=2pt] {\small $n_{2}$} (4,3.5);
\draw[decoration={brace,raise=2pt},decorate] (2,2.5) -- node[right=2pt] {\small $n_{r-1}$} (2,1.5);
\draw[decoration={brace,raise=2pt},decorate] (1,1.5) -- node[right=2pt] {\small $n_{r}$} (1,0);

\draw [lightgray,ultra thin] (4,5.5) -- (4,4.7);
\draw [lightgray,ultra thin] (4.3,5.5) -- (4.3,4.7);
\draw [lightgray,ultra thin] (4.6,5.5) -- (4.6,4.7);
\draw [lightgray,ultra thin] (4.9,5.5) -- (4.9,4.7);
\draw [lightgray,ultra thin] (5.2,5.5) -- (5.2,4.7);
\draw [lightgray,ultra thin] (5.5,5.5) -- (5.5,4.7);
\draw [lightgray,ultra thin] (5.8,5.5) -- (5.8,4.7);
\draw [lightgray,ultra thin] (3.7,5.5) -- (3.7,3.7);
\draw [lightgray,ultra thin] (3.4,5.5) -- (3.4,3.7);
\draw [lightgray,ultra thin] (3.1,5.5) -- (3.1,3.7);
\draw [lightgray,ultra thin] (2.8,5.5) -- (2.8,2.7);
\draw [lightgray,ultra thin] (2.5,5.5) -- (2.5,2.7);
\draw [lightgray,ultra thin] (2.2,4.7) -- (2.2,2.7);
\draw [lightgray,ultra thin] (1.9,4.7) -- (1.9,2.7);
\draw [lightgray,ultra thin] (1.6,4.7) -- (1.6,1.7);
\draw [lightgray,ultra thin] (1.3,5.5) -- (1.3,3.1);
\draw [lightgray,ultra thin] (1,5.5) -- (1,3.3);
\draw [lightgray,ultra thin] (0.7,5.5) -- (0.7,3.3);
\draw [lightgray,ultra thin] (0.4,5.5) -- (0.4,2.4);
\draw [lightgray,ultra thin] (1.3,2.7) -- (1.3,1.6);
\draw [lightgray,ultra thin] (1,2.7) -- (1,1.5);
\draw [lightgray,ultra thin] (0.7,2.7) -- (0.7,2);
\draw [lightgray,ultra thin] (0.4,1.7) -- (0.4,0);
\draw [lightgray,ultra thin] (0.7,1.7) -- (0.7,0);

\draw [lightgray,ultra thin] (0.1,0.3) -- (0.9,0.3);
\draw [lightgray,ultra thin] (0.1,0.6) -- (0.9,0.6);
\draw [lightgray,ultra thin] (0.1,0.9) -- (0.9,0.9);
\draw [lightgray,ultra thin] (0.1,1.2) -- (0.9,1.2);
\draw [lightgray,ultra thin] (0.1,1.5) -- (0.9,1.5);

\draw [lightgray,ultra thin] (0.9,1.9) -- (1.9,1.9);
\draw [lightgray,ultra thin] (0.9,2.2) -- (1.9,2.2);
\draw [lightgray,ultra thin] (0.1,2.5) -- (1.9,2.5);

\draw [lightgray,ultra thin] (0.1,3.5) -- (2.9,3.5);
\draw [lightgray,ultra thin] (1.1,3.2) -- (2.9,3.2);
\draw [lightgray,ultra thin] (1.8,2.9) -- (2.9,2.9);

\draw [lightgray,ultra thin] (0.1,4.5) -- (3.9,4.5);
\draw [lightgray,ultra thin] (0.1,4.16) -- (3.9,4.16);
\draw [lightgray,ultra thin] (0.1,3.85) -- (3.9,3.85);

\draw [lightgray,ultra thin] (0.1,4.85) -- (1.4,4.85);
\draw [lightgray,ultra thin] (0.1,5.2) -- (1.4,5.2);
\draw [lightgray,ultra thin] (2.3,5.2) -- (6,5.2);
\draw [lightgray,ultra thin] (2.3,4.85) -- (6,4.85);

\draw [->,thick] (7.4,3) -- node[above=2pt]{$\rho$} (10.4,3);

\begin{scope}[shift={(12,0.3)}]

\draw (9,3) node{{$= \rho(\lambda)$}};
\draw (0,2.5) -- (0,-0.6) -- (1,-0.6) -- (1,1.5) -- (2,1.5) -- (2,2.5);
\draw [densely dotted] (0,2.5) -- (0,3.5);
\draw [densely dotted] (2,2.5) -- (3,2.5) -- (3,3.5);
\draw (3,3.5) -- (4,3.5) -- (4,4.5) -- (5.5,4.5) -- (5.5,5.5) -- (0,5.5) -- (0,3.5);

\draw [thick, blue] (0,5.5) -- (5.5,5.5); 
\draw [thick, red] (0,5.5) -- (0,3.5); 
\draw [thick, red, densely dotted] (0,2.5) -- (0,3.5);
\draw [thick, red] (0,-0.6) -- (0,2.5);

\draw[decoration={brace,raise=2pt},decorate] (0,5.5) -- node[above=2pt] {\small $n_1 + \dots + n_r$} (5.5,5.5);
\draw[decoration={brace,raise=2pt},decorate] (0,4.5) -- node[above=2pt] {\tiny $n_1 + \dots + n_{r-1}$} (4,4.5);
\draw[decoration={brace,raise=2pt},decorate] (0,2.5) -- node[above=2pt] {\tiny $n_1+n_2$} (2,2.5);
\draw[decoration={brace,raise=2pt},decorate] (0,1.5) -- node[above=2pt] {\tiny $n_1$} (1,1.5);

\draw[decoration={brace,raise=2pt},decorate] (5.5,5.5) -- node[right=2pt] {\small $m_r$} (5.5,4.5);
\draw[decoration={brace,raise=2pt},decorate] (4,4.5) -- node[right=2pt] {\small $m_{r-1}-m_r$} (4,3.5);
\draw[decoration={brace,raise=2pt},decorate] (2,2.5) -- node[right=2pt] {\small $m_2-m_3$} (2,1.5);
\draw[decoration={brace,raise=2pt},decorate] (1,1.5) -- node[right=2pt] {\small $m_1-m_2$} (1,-0.6);

\draw [lightgray,ultra thin] (4,5.5) -- (4,4.7);
\draw [lightgray,ultra thin] (4.3,5.5) -- (4.3,4.7);
\draw [lightgray,ultra thin] (4.6,5.5) -- (4.6,4.7);
\draw [lightgray,ultra thin] (4.9,5.5) -- (4.9,4.7);
\draw [lightgray,ultra thin] (5.2,5.5) -- (5.2,4.7);
\draw [lightgray,ultra thin] (3.7,5.5) -- (3.7,3.7);
\draw [lightgray,ultra thin] (3.4,5.5) -- (3.4,3.7);
\draw [lightgray,ultra thin] (3.1,5.5) -- (3.1,3.7);
\draw [lightgray,ultra thin] (2.8,5.5) -- (2.8,2.7);
\draw [lightgray,ultra thin] (2.5,5.5) -- (2.5,2.7);
\draw [lightgray,ultra thin] (2.2,4.7) -- (2.2,2.7);
\draw [lightgray,ultra thin] (1.9,4.7) -- (1.9,2.7);
\draw [lightgray,ultra thin] (1.6,4.7) -- (1.6,1.7);
\draw [lightgray,ultra thin] (1.3,5.5) -- (1.3,3.1);
\draw [lightgray,ultra thin] (1,5.5) -- (1,3.3);
\draw [lightgray,ultra thin] (0.7,5.5) -- (0.7,3.3);
\draw [lightgray,ultra thin] (0.4,5.5) -- (0.4,2.4);
\draw [lightgray,ultra thin] (1.3,2.7) -- (1.3,1.6);
\draw [lightgray,ultra thin] (1,2.7) -- (1,1.5);
\draw [lightgray,ultra thin] (0.7,2.7) -- (0.7,2);
\draw [lightgray,ultra thin] (0.4,1.7) -- (0.4,-0.5);
\draw [lightgray,ultra thin] (0.7,1.7) -- (0.7,-0.5);
\draw [lightgray,ultra thin] (0.1,-0.3) -- (0.9,-0.3);
\draw [lightgray,ultra thin] (0.1,0) -- (0.9,0);
\draw [lightgray,ultra thin] (0.1,0.3) -- (0.9,0.3);
\draw [lightgray,ultra thin] (0.1,0.6) -- (0.9,0.6);
\draw [lightgray,ultra thin] (0.1,0.9) -- (0.9,0.9);
\draw [lightgray,ultra thin] (0.1,1.2) -- (0.9,1.2);
\draw [lightgray,ultra thin] (0.1,1.5) -- (0.9,1.5);

\draw [lightgray,ultra thin] (0.9,1.9) -- (1.9,1.9);
\draw [lightgray,ultra thin] (0.9,2.2) -- (1.9,2.2);
\draw [lightgray,ultra thin] (0.1,2.5) -- (1.9,2.5);

\draw [lightgray,ultra thin] (0.1,3.5) -- (2.9,3.5);
\draw [lightgray,ultra thin] (1.1,3.2) -- (2.9,3.2);
\draw [lightgray,ultra thin] (1.8,2.9) -- (2.9,2.9);

\draw [lightgray,ultra thin] (0.1,4.5) -- (3.9,4.5);
\draw [lightgray,ultra thin] (0.1,4.16) -- (3.9,4.16);
\draw [lightgray,ultra thin] (0.1,3.85) -- (3.9,3.85);

\draw [lightgray,ultra thin] (0.1,4.85) -- (1.4,4.85);
\draw [lightgray,ultra thin] (0.1,5.2) -- (1.4,5.2);
\draw [lightgray,ultra thin] (2.3,5.2) -- (5.4,5.2);
\draw [lightgray,ultra thin] (2.3,4.85) -- (5.4,4.85);
\end{scope}
\end{tikzpicture}
	\end{center}
	\vspace{-0.7cm}
\end{figure}
Since we sum over all elements in $\operatorname{Part}_r(N)$ we obtain 
\begin{align*}
		g\bi{k_1,\dots,k_r}{d_1,\dots,d_r} = \sum_{N>0} \left( \sum_{\lambda \in \operatorname{Part}_r(N)} f(\lambda) \right) q^N =  \sum_{N>0} \left( \sum_{\lambda \in \operatorname{Part}_r(N)} f(\rho(\lambda)) \right) q^N\,.
\end{align*}
The $f(\rho(\lambda))$ is again a polynomial in $m_1,\dots,m_r,n_1,\dots,n_r$ and therefore the right-hand side of above equation can be written again as a linear combination of the $q$-series $g$ in the same depth and weight. We call the linear relation obtained from this \emph{partition relation}. These linear relations can be written done explicitly in all depths, and we give the explicit formula for depth one and two in the following.
\begin{proposition}[Partition relation]\label{prop:partrel}
\begin{enumerate}[i)]
    \item For $k\geq 1, d\geq 1$ we have
    \begin{align*}
        g\bi{k}{d} =   \frac{d!}{(k-1)!} 	g\bi{d+1}{k-1}\,.
    \end{align*}
    \item For $k_1,k_2\geq 1, d_1,d_2\geq 0$ we have
    \begin{align*}
        g\bi{k_1,k_2}{d_1,d_2} = \sum_{\substack{0 \leq a \leq d_1\\0 \leq b \leq k_2-1}} \frac{(-1)^b}{a!\, b!} \frac{d_1!}{(k_1-1)!} \frac{(d_2+a)!}{(k_2-1-b)!} \,\,g\bi{d_2+1+a,\,d_1+1-a}{k_2-1-b,\,k_1-1+b}\,.
    \end{align*}
\end{enumerate}
\end{proposition}

\begin{example}\label{ex:shuffleexample}Combining the stuffle product (Proposition \ref{prop:stufflebig}) and the partition relation (Proposition \ref{prop:partrel}) gives us another way of expressing the product of two $q$-series $g$:
    	\begin{align*}
		g(2)g(3) &= g\bi{2}{0}  g\bi{3}{0} \overset{\text{Prop.} \ref{prop:partrel}}{=} \frac{1}{2} g\bi{1}{1}  g\bi{1}{2} \\
		&\overset{\text{Prop.} \ref{prop:stufflebig}}{=} \frac{1}{2} \left( g\bi{1,1}{1,2}  + g\bi{1,1}{2,1} + g\bi{2}{3}  - g\bi{1}{3}   \right) \\
		&\overset{\text{Prop.} \ref{prop:partrel}}{=}  g\bi{2,3}{0,0} +  3 g\bi{3,2}{0,0} + 6 g\bi{4,1}{0,0}+ 3 g\bi{4}{1}  - 3 g\bi{4}{0}.
	\end{align*}
	Using $q \frac{d}{dq} g\bi{3}{0} = 3 g\bi{4}{1} $ we obtain 
	\begin{align*}
		g(2)g(3) &=g(2,3) + 3 g(3,2)+ 6 g(4,1)-3 g(4)+q\frac{d}{dq} g(3)\,.
	\end{align*}
	This formula looks similar to the shuffle product of the corresponding zeta values, which is given by 
	\begin{align*}
	    \zeta(2) \zeta(3) = \zeta(2,3) + 3 \zeta(3,2) + 6 \zeta(4,1)\,.
	\end{align*}
	But now we do not just have an additional lower weight term $3 g\bi{4}{0}$, but also the term $3 g\bi{4}{1}$, which is of weight $5$. Though these terms vanish after multiplying with $(1-q)^5$ and sending $q\rightarrow 1$, this additional weight $5$ term is the reason why there is no realization of the formal double zeta space $\dz_k$ with $Z_k \mapsto G_k$ and where the image of $P_{k_1,k_2}$ is given exactly by products of Eisenstein series. 
\end{example}

\begin{proposition}[Shuffle product analogue] \label{prop:shufflebi}
     For $k_1,k_2 \geq 1, d_1,d_2 \geq 0$ we have
\begin{align*}
		g\bi{k_1}{d_1} g\bi{k_2}{d_2} = &\sum_{\substack{l_1+l_2=k_1+k_2\\ e_1+e_2=d_1+d_2}} \left(\binom{l_1-1}{k_1-1}\binom{d_1}{e_1}(-1)^{d_1-e_1} +   \binom{l_1-1}{k_2-1}\binom{d_2}{e_1} (-1)^{d_2-e_1}  \right) g\bi{l_1,l_2}{e_1,e_2} \\&+\frac{d_1! d_2!}{(d_1+d_2+1)!}\binom{k_1+k_2-2}{k_1-1}g\bi{k_1+k_2-1}{d_1+d_2+1} \\
		&+d_1! d_2! \binom{k_1+k_2-2}{k_1-1}  \sum_{j=1}^{d_1+d_2+1}  \frac{\lambda^j_{d_1+1,d_2+1}}{(j-1)!} \, g\bi{k_1+k_2-1}{j-1}\,.
\end{align*}
\end{proposition}
\begin{proof}
Similar as in Example \ref{ex:shuffleexample} we first use Proposition \ref{prop:partrel} i), then Proposition  \ref{prop:stufflebig} to evaluate the product and then use again Proposition \ref{prop:partrel} i) and ii).
\end{proof}

\section{Formal double Eisenstein space}
In this section, we will introduce the formal double Eisenstein space which will be discussed in more detail in \cite{BKM}. This space will be spanned by formal symbols satisfying the same relations as the $q$-series $g$ except that we just consider the homogeneous weight parts and ignore the lower weight terms. The motivation behind this is that, for example, in depth one, the main objects we are interested in are not the $g(k)$ but the Eisenstein series $\widetilde{G}_k=-\frac{B_k}{2k!}+g(k)$ which do not satisfy algebraic relations in mixed weight. 
\begin{definition} 
	We define for $K\geq 1$ the \emph{formal double Eisenstein space} of weight $K$ as 
	\begin{align*}
		\de_K = \quotient{\Big \langle G\bi{k}{d} , G\bi{k_1,k_2}{d_1,d_2}, P\bi{k_1,k_2}{d_1,d_2} \mid \substack{ k+d = k_1 + k_2+d_1+d_2 =K \\k, k_1,k_2 \geq 1,\, d, d_1,d_2\geq 0 }\Big \rangle_\Q}{\eqref{eq:derel} }    \,,
	\end{align*}
	where we divide out the following relations 
	{\small
	\begin{align}\begin{split} \label{eq:derel}
			P\bi{k_1,k_2}{d_1,d_2}  &=  G\bi{k_1,k_2}{d_1,d_2} +G\bi{k_2,k_1}{d_2,d_1} +G\bi{k_1+k_2}{d_1+d_2} \\
			&=  \sum_{\substack{l_1+l_2=k_1+k_2\\ e_1+e_2=d_1+d_2\\l_1,l_2\geq 1, e_1,e_2\geq 0}} \left(\binom{l_1-1}{k_1-1}\binom{d_1}{e_1}(-1)^{d_1-e_1} +   \binom{l_1-1}{k_2-1}\binom{d_2}{e_1} (-1)^{d_2-e_1}  \right) G\bi{l_1,l_2}{e_1,e_2} \\&+\frac{d_1! d_2!}{(d_1+d_2+1)!}\binom{k_1+k_2-2}{k_1-1}G\bi{k_1+k_2-1}{d_1+d_2+1}\,.
		\end{split}
	\end{align}}
\end{definition}
The reason why this space can be seen as a generalization of the formal double zeta space $\dz_k$ is that we have the following embedding of $\dz_k$ in $\de_k$.
\begin{proposition} \label{prop:dktoek}
	For all  $k\geq 1$ the following gives a $\Q$-linear map $\dz_k \rightarrow \de_k$
	\begin{align*}
		Z_k &\longmapsto G\bi{k}{0} - \delta_{k,2} G\bi{2}{0} \,,\\ \quad Z_{k_1,k_2} &\longmapsto G\bi{k_1,k_2}{0,0} + \frac{1}{2} \left(  \delta_{k_2,1}G\bi{k_1}{1} -  \delta_{k_1,1} G\bi{k_2}{1} +   \delta_{k_1,2} G\bi{k_2+1}{1} \right)  \,,\\
		P_{k_1,k_2} &\longmapsto 	P\bi{k_1,k_2}{0,0}+ \frac{1}{2} \left(    \delta_{k_1,2} G\bi{k_2+1}{1} + \delta_{k_2,2} G\bi{k_1+1}{1} \right) - \delta_{k_1 \cdot k_2,1} G\bi{2}{0} \,.
	\end{align*}
\end{proposition}
\begin{proof}
We just need to check that this map is well-defined. This can be done by applying this map to the defining equation \eqref{eq:dzrel} of $\dz_k$ and then observe that, after rearranging some terms and using $G\bi{2}{0}=G\bi{1}{1}$, the result is exactly  \eqref{eq:derel} with $d_1=d_2=0$.
\end{proof}

Due to Proposition \ref{prop:dktoek} any relation in the formal double zeta space also gives a relation in the formal double Eisenstein space. We will use this in the proof of Theorem \ref{thm:relprodandg} below. First we will be interested in realizations of the space $\de_k$, by which we again mean elements in $\Hom_\Q(\de_k,A)$ for some $\Q$-vector space $A$. Also notice that by Proposition \ref{prop:dktoek} any realization of $\de_k$ gives a realization of $\dz_k$. To write down explicit realizations for all weight $k$ we will consider generating series defined as follows
{\small 
	\begin{align*}
		\Genz_1\!\bi{X_1}{Y_1} := &\sum_{\substack{k_1\geq 1\\d_1 \geq 0}} G\bi{k_1}{d_1} X_1^{k_1-1}   \frac{Y_1^{d_1}}{d_1!}  \,, \qquad 
		\Genz_2\!\bi{X_1,X_2}{Y_1,Y_2} := \sum_{\substack{k_1,k_2\geq 1\\d_1,d_2 \geq 0}} G\bi{k_1,k_2}{d_1,d_2} X_1^{k_1-1}  X_2^{k_2-1} \frac{Y_1^{d_1}}{d_1!}  \frac{Y_2^{d_2}}{d_2!}\,, \\
		\gp{X_1,X_2}{Y_1,Y_2} := &\sum_{\substack{k_1,k_2\geq 1\\d_1,d_2 \geq 0}} P\bi{k_1,k_2}{d_1,d_2} X_1^{k_1-1}  X_2^{k_2-1} \frac{Y_1^{d_1}}{d_1!}  \frac{Y_2^{d_2}}{d_2!}\,.	
\end{align*}} 
With this  \eqref{eq:derel} can be written as
	\begin{align}\label{eq:qdsh}
		\begin{split}
			\gp{X_1,X_2}{Y_1,Y_2}&=  \Genz_2\!\bi{X_1,X_2}{Y_1,Y_2} + \Genz_2\!\bi{X_2,X_1}{Y_2,Y_1} + \frac{\Genz_1\!\bi{X_1}{Y_1+Y_2} -\Genz_1\!\bi{X_2}{Y_1+Y_2}}{X_1-X_2} \\
			&= \Genz_2\!\bi{X_1+X_2, X_2}{Y_1, Y_2-Y_1}+\Genz_2\!\bi{X_1+X_2,X_1}{Y_2, Y_1-Y_2}  + \frac{\Genz_1\!\bi{X_1+X_2}{Y_1}-\Genz_1\!\bi{X_1+X_2}{Y_2}}{Y_1-Y_2} \,.
		\end{split}
	\end{align}
Finding a realization for all $k\geq 1$ of $\de_k$ in some space $A$ is therefore equivalent of finding power series with coefficients in $A$ satisfying \eqref{eq:qdsh}.
In \cite{BKM} a systematic way of finding power series satisfying these equations will be presented. In the following we give one particular example of such a family of power series without giving its precise origin. For this we define 
\begin{align*}
\benz_1\bi{X}{Y} &:= - \frac{1}{4}\left( \coth\left( \frac{X}{2}\right)+ \coth\left( \frac{Y}{2}\right)  \right) + \frac{1}{2}\left(\frac{1}{X} + \frac{1}{Y} \right) 
\end{align*}
and with this
\begin{align*}
	R^\ast_\benz\bi{X_1,X_2}{Y_1,Y_2} &:= \frac{\benz_1\bi{X_1}{Y_1+Y_2} -\benz_1\bi{X_2}{Y_1+Y_2}}{X_1-X_2}\,,\quad R^\shuffle_\benz\bi{X_1,X_2}{Y_1,Y_2} :=\frac{\benz_1\bi{X_1+X_2}{Y_1}-\benz_1\bi{X_1+X_2}{Y_2}}{Y_1-Y_2}\,,\\	
\penz_\benz\bi{X_1,X_2}{Y_1,Y_2}&:=\benz_1\bi{X_1}{Y_1}\benz_1\bi{X_2}{Y_2}\,,\\
\benz_2\bi{X_1,X_2}{Y_1,Y_2}&:=\frac{1}{3} \penz_\benz\bi{X_1,X_2}{Y_1,Y_2} + \frac{1}{3}\penz_\benz\bi{X_1-X_2,X_2}{Y_1,Y_1+Y_2} \\
&- \frac{5}{12} R^\shuffle_\benz\bi{X_1-X_2,X_2}{Y_1,Y_1+Y_2} -  \frac{1}{12} R^\shuffle_\benz\bi{-X_2,X_1}{-Y_2,Y_1} + \frac{1}{4} R^\shuffle_\benz\bi{X_2-X_1,X_1}{Y_2,Y_1+Y_2}\\
&- \frac{5}{12} R^\ast_\benz\bi{X_1,X_2}{Y_1,Y_2} - \frac{1}{12} R^\ast_\benz\bi{X_2-X_1,X_2}{-Y_1,Y_1+Y_2} + \frac{1}{4} R^\ast_\benz\bi{X_1-X_2,X_1}{-Y_2,Y_1+Y_2} \,.
\end{align*}
Writing $\coth$ in terms of exponential functions one then checks by direct calculation that 
\begin{align}\label{eq:betadsh}
	\begin{split}
		\penz_\benz\bi{X_1,X_2}{Y_1,Y_2}&=  \benz_2\bi{X_1,X_2}{Y_1,Y_2} + \benz_2\bi{X_2,X_1}{Y_2,Y_1} + \frac{\benz_1\bi{X_1}{Y_1+Y_2} -\benz_1\bi{X_2}{Y_1+Y_2}}{X_1-X_2} \\
		&= \benz_2\bi{X_1+X_2, X_2}{Y_1, Y_2-Y_1}+\benz_2\bi{X_1+X_2,X_1}{Y_2, Y_1-Y_2}  + \frac{\benz_1\bi{X_1+X_2}{Y_1}-\benz_1\bi{X_1+X_2}{Y_2}}{Y_1-Y_2} \,,
	\end{split}
\end{align}
i.e. this gives a realization of $\de_k$ in $\Q$ for any $k\geq 1$ by taking the coefficients of the powers series $\benz_1,\benz_2$ and $\penz_\benz$ for the images of $G\bi{k}{d},G\bi{k_1,k_2}{d_1,d_2}$ and $P\bi{k_1,k_2}{d_1,d_2}$ respectively. In depth one this realization satisfies $G\bi{k}{0} \mapsto -\frac{B_k}{2k!}$ for $k>1$ and together with Proposition \ref{prop:dktoek} this gives a realization of $\dz_k$ in $\Q$ which is called the \emph{Bernoulli realization} in \cite{GKZ}. 
The above realization will give the constant term of the Eisenstein realization which we will introduce now. For this we first define the generating series of the $q$-series $g$
\begin{align*}
	\mathfrak{g}_1\bi{X}{Y} :=\sum_{\substack{k\geq 1\\d\geq 0}} g\bi{k}{d}X^{k-1} \frac{Y^d}{d!}\,,\quad 
		\mathfrak{g}_2\bi{X_1, X_2}{Y_1, Y_2} :=\sum_{\substack{k_1, k_2\geq 1\\d_1, d_2\geq 0}} g\bi{k_1, k_2}{d_1, d_2}X_1^{k_1-1}X_2^{k_2-1}  \frac{Y_1^{d_1}}{d_1!} \frac{Y_2^{d_2}}{d_2!} \,.
\end{align*}
Using these generating series the stuffle product analogue (Proposition \ref{prop:stufflebig}), the partition relation (Proposition \ref{prop:partrel}) and the shuffle product analogue (Proposition \ref{prop:shufflebi}) can be expressed as follows.
\begin{lemma}\label{lem:algstruct}
The generating series $\genz_1$ and $\genz_2$  satisfy the following relations
\begin{enumerate}[i)]
	\item Stuffle product analogue:
	{\small
\begin{align*}
	\begin{split}
		\genz_1\bi{X_1}{Y_1} \genz_1\bi{X_2}{Y_2}  = &\,\genz_2\bi{X_1, X_2}{Y_1, Y_2} +  \genz_2\bi{X_2, X_1}{Y_2, Y_1}+\frac{1}{X_1-X_2} \left(\genz_1\bi{X_1}{Y_1+Y_2} -   \genz_1\bi{X_2}{Y_1+Y_2}  \right)\\
		&+ \Big(\benz_1\bi{X_2-X_1}{Y_1+Y_2} -\benz_1\bi{X_1-X_2}{Y_1+Y_2}   \Big) \left(\genz_1\bi{X_1}{Y_1+Y_2} -   \genz_1\bi{X_2}{Y_1+Y_2}  \right)\\
		& -\frac{1}{2}\left( \genz_1\bi{X_1}{Y_1+Y_2}+  \genz_1\bi{X_2}{Y_1+Y_2}  \right) \,.
	\end{split}
\end{align*}}
	\item Partition relation:
	\begin{align*}
	\genz_1\bi{X_1}{Y_1} = \genz_1\bi{Y_1}{X_1} \,,\qquad \genz_2\bi{X_1, X_2}{Y_1, Y_2}  = \genz_2\bi{Y_1+Y_2, Y_1}{X_2, X_1-X_2}\,.
	\end{align*}
\item Shuffle product analogue:
	{\small
\begin{align*}
\begin{split}
\genz_1\bi{X_1}{Y_1} \genz_1\bi{X_2}{Y_2}  = &\,\genz_2\bi{X_1+X_2, X_1}{Y_2, Y_1-Y_2} + \genz_2\bi{X_1+X_2, X_2}{Y_1, Y_2-Y_1}  + \frac{1}{Y_1 - Y_2} \left(\genz_1\bi{X_1+X_2}{Y_1} -   \genz_1\bi{X_1+X_2}{Y_2}  \right)\\
&+ \Big(\benz_1\bi{Y_2-Y_1}{X_1+X_2} -\benz_1\bi{Y_1-Y_2}{X_1+X_2} \Big) \ \left(\genz_1\bi{X_1+X_2}{Y_1} -   \genz_1\bi{X_1+X_2}{Y_2}  \right)\\
&-\frac{1}{2}\left( \genz_1\bi{X_1+X_2}{Y_1}+  \genz_1\bi{X_1+X_2}{Y_2}    \right) \,.
\end{split}
\end{align*}}
\end{enumerate}
\end{lemma}
\begin{proof}
These are exactly the statements of Proposition \ref{prop:stufflebig}, \ref{prop:partrel} and \ref{prop:shufflebi} written down in terms of generating series. 
\end{proof}
Using Lemma \ref{lem:algstruct} together with \eqref{eq:betadsh} one can construct a realization of $\de_k$ for all $k\geq 1$ with $G\bi{k}{0} \mapsto -\frac{B_k}{2k!} + g(k) = \widetilde{G}_k$ for $k>1$ such that $P\bi{k_1,k_2}{0,0}\mapsto \widetilde{G}_{k_1} \widetilde{G}_{k_2}$. The hard part of this is to find the correct image in depth two. There exist a general construction behind this, which is inspired by the Fourier expansion of double Eisenstein series, and which will be explained in detail in \cite{BB}.
\begin{theorem} \label{thm:eisensteinreal}The power series 
\begin{align*}
	\eenz_1\bi{X_1}{Y_1} &:=\benz_1\bi{X_1}{Y_1} + \genz_1\bi{X_1}{Y_1}\,,\\
	\eenz_2\bi{X_1, X_2}{Y_1, Y_2} &:=\benz_2\bi{X_1, X_2}{Y_1, Y_2} - \benz_1\bi{X_1-X_2}{Y_2}\genz_1\bi{X_1}{Y_1+Y_2} -\frac{1}{2} \genz_1\bi{X_1}{Y_1+Y_2} \\
	&+\benz_1\bi{X_2}{Y_2}  \genz_1\bi{X_1}{Y_1} +\benz_1\bi{X_1-X_2}{Y_1}\genz_1\bi{X_2}{Y_1+Y_2} +\genz_2\bi{X_1, X_2}{Y_1, Y_2}
\end{align*} 
and $\penz_\eenz\bi{X_1,X_2}{Y_1,Y_2}=\eenz_1\!\bi{X_1}{Y_1}\eenz_1\!\bi{X_2}{Y_2}$ satisfy 
\begin{align*}
	\begin{split}
		\penz_\eenz\bi{X_1,X_2}{Y_1,Y_2}&=  \eenz_2\bi{X_1,X_2}{Y_1,Y_2} + \eenz_2\bi{X_2,X_1}{Y_2,Y_1} + \frac{\eenz_1\bi{X_1}{Y_1+Y_2} -\eenz_1\bi{X_2}{Y_1+Y_2}}{X_1-X_2} \\
		&= \eenz_2\bi{X_1+X_2, X_2}{Y_1, Y_2-Y_1}+\eenz_2\bi{X_1+X_2,X_1}{Y_2, Y_1-Y_2}  + \frac{\eenz_1\bi{X_1+X_2}{Y_1}-\eenz_1\bi{X_1+X_2}{Y_2}}{Y_1-Y_2} \,.
	\end{split}
\end{align*}
\end{theorem}
\begin{proof}
This follows by using Lemma \ref{lem:algstruct} together with \eqref{eq:betadsh}.
\end{proof}
From Theorem \ref{thm:eisensteinreal}  we obtain a realization of $\de_k$ in $\Q[[q]]$ for any $k\geq 1$ by taking the coefficients of the powers series $\eenz_1,\eenz_2$ and $\penz_\eenz$ for the images of $G\bi{k}{d},G\bi{k_1,k_2}{d_1,d_2}$ and $P\bi{k_1,k_2}{d_1,d_2}$ respectively.  We call this realization the \emph{Eisenstein realization}. In particular this realization satisfies
\begin{align}\label{eq:eisensteinreal}\begin{split}
		G\bi{k}{d}&\longmapsto \frac{(k-d-1)!}{(k-1)!} \left(q \frac{d}{dq}\right)^d \widetilde{G}_{k-d},\qquad (k>d\geq 0)\\
	P\bi{k_1,k_2}{0,0}&\longmapsto \widetilde{G}_{k_1} \widetilde{G}_{k_2} ,\qquad (k_1,k_2\geq 1)\,,
\end{split}
\end{align}
where $\widetilde{G}_{k}$ is defined by \eqref{eq:defgtilde} for $k\geq 2$ and $\widetilde{G}_{1}:=g(1)$.
\begin{remark} We call the coefficients of $\eenz_1$ and $\eenz_2$ \emph{combinatorial multiple Eisenstein series}. In \cite{BB} these objects will be considered for higher depths and the construction in Theorem \ref{thm:eisensteinreal} will be explained in detail. 
\end{remark}
\section{Application}
In the formal double Eisenstein space a lot explicit relations can be proven by using the defining relations \eqref{eq:derel}. A lot of these relations will be explained in \cite{BKM} and we will just mention the following Theorem which follows from a result proven in \cite{B2} for the formal double zeta space.
\begin{theorem}\label{thm:relprodandg}
For all $k_1,k_2\geq 1$ with $k=k_1+k_2\geq 4$ even we have 
{\small
\begin{align} \label{eq:relprodandg}\begin{split}
		\frac{1}{2}\left(  \binom{k_1+k_2}{k_2} - (-1)^{k_1}\right) G\bi{k}{0} =  &\sum_{\substack{j=2\\j \text{even}}}^{k-2} \left( \binom{k-j-1}{k_1-1} + \binom{k-j-1}{k_2-1} - \delta_{j,k_1} \right)  P\bi{j,k-j}{0,0}   \\
		& + \frac{1}{2} \left( \binom{k-3}{k_1-1} + \binom{k-3}{k_2-1}  + \delta_{k_1,1} + \delta_{k_2,1} \right) G\bi{k-1}{1}\,. 
		\end{split}
\end{align}}
\end{theorem}
\begin{proof}
By \cite[Theorem 4.9]{B2} the following relation in the formal double zeta space holds
\begin{align*}
\frac{1}{2}\left(  \binom{k_1+k_2}{k_2} - (-1)^{k_1}\right) Z_k =  \sum_{\substack{j=2\\j \text{even}}}^{k-2} \left( \binom{k-j-1}{k_1-1} + \binom{k-j-1}{k_2-1} - \delta_{j,k_1} \right)  P_{j,k-j}    \,.
\end{align*}
Applying the linear map from Proposition \ref{prop:dktoek} to this gives exactly \eqref{eq:relprodandg}.
\end{proof}
 In \cite{BKM} a generalization of Theorem \ref{thm:relprodandg} will be presented. This theorem has nice special cases which we will give now.  
\begin{corollary} \label{cor:mfprod}
	\begin{enumerate}[i)]
		\item For even $k\geq 4$ we have
		\begin{align*}
	 G\bi{k-1}{1}	 =   \frac{k+1}{2} G\bi{k}{0} + \sum_{\substack{k_1+k_2=k \\ k_1, k_2\geq 2 \text{ even} }}  P\bi{k_1,k_2}{0,0} \,.
		\end{align*}
		\item 	 For all even $k\geq 6$ we have
		\begin{align*}
			\frac{(k+1)(k-1)(k-6)}{12}	 G\bi{k}{0} =  \sum_{\substack{k_1+k_2 = k\\k_1,k_2\geq 4 \text{ even}}} (k_1-1)(k_2-1)   \,P\bi{k_1,k_2}{0,0}\,.
		\end{align*}
	\end{enumerate}
\end{corollary}
\begin{proof}
Part i) is the $k_1=1$ case of Theorem \ref{thm:relprodandg}. Part ii) follows by considering $k\!-\!3$-times the $(k_1,k_2)=(k-2,2)$ case and then subtracting $2$-times the $(k_1,k_2)=(k-3,3)$ case.
\end{proof}

\begin{example}
    \begin{enumerate}[i)]
    \item We have 
        \begin{align*}
        	G\bi{8}{0}  = \frac{6}{7} G_4^2,\,\quad  G\bi{10}{0}  = \frac{10}{11} P\bi{4,6}{0,0} \,.
        \end{align*}
        \item Analogues of the Ramanujan differential equations are satisfied 
		\begin{align*}
			2 G\bi{3}{1} &= 5 G\bi{4}{0} - 2 P\bi{2,2}{0,0}\,,\qquad 
			4 G\bi{5}{1} = 8 G\bi{6}{0} - 14 P\bi{2,4}{0,0}\,,\\
			6 G\bi{7}{1} &= \frac{120}{7} P\bi{4,4}{0,0} - 12 P\bi{2,6}{0,0}\,.
		\end{align*}
    \end{enumerate}
\end{example}

Applying the Eisenstein realization \eqref{eq:eisensteinreal} we get the well-known relations $\widetilde{G}_8 = \frac{6}{7} \widetilde{G}_4^2$, $\widetilde{G}_{10} = \frac{10}{11} \widetilde{G}_4  \widetilde{G}_6$ as well as the Ramanujan differential equations
\begin{align*}
	q \frac{d}{dq} \widetilde{G}_{2} = 5 \widetilde{G}_{4} - 2 \widetilde{G}_{2}^{2}\,,\quad
	q \frac{d}{dq} \widetilde{G}_{4} = 8 \widetilde{G}_{6} - 14 \widetilde{G}_{2}\widetilde{G}_{4}\,,\quad
	q \frac{d}{dq} \widetilde{G}_{6} = \frac{120}{7} \widetilde{G}_{4}^2 - 12 \widetilde{G}_{2} \widetilde{G}_{6}\,.
\end{align*}

\begin{remark}
\begin{enumerate}[i)]
    \item These proofs of identities among quasi-modular forms are purely combinatorial and are similar to those done in \cite{S}. Applying the Eisenstein realization to Theorem \ref{thm:eisensteinreal} can also be compared to the relations proven in \cite[Theorem 1]{HST}.
    \item In \cite{BIM} everything done here in depth two will be generalized to arbitrary depths by introducing the algebra of \emph{formal multiple Eisenstein series}. In this algebra, one can also define a subalgebra of \emph{formal (quasi-)modular forms} for which (almost) all standard relations known for modular forms can be proven on a formal level.
\end{enumerate}
\end{remark}

\end{document}